\documentclass[10pt]{amsart}

\usepackage{hyperref}

\usepackage{amssymb}
\usepackage{amsmath}
\usepackage{graphicx}
\usepackage{tikz}
\usepackage{amsthm}

\let\newpf\proof \let\proof\relax

\def\bm{\begin{matrix}}
\def\em{\end{matrix}}

\newcommand{\bt}{\begin{thm}}
\newcommand{\et}{\end{thm}}

\newcommand{\bl}{\begin{lemma}}
\newcommand{\el}{\end{lemma}}

\newcommand{\beq}{\begin{eqnarray}}
\newcommand{\eeq}{\end{eqnarray}}

\def\be{\begin{equation}}
\def\ee{\end{equation}}

\def\ba{{\begin{align}}}
\def\ea{{\end{align}}}

\def\0{{\mathbf 0}}

\newtheorem{thm}{Theorem}[section]

\newtheorem{lemma}[thm]{Lemma}

\newtheorem{prop}[thm]{Proposition}

\theoremstyle{remark}
\newtheorem{rem}{Remark}[section]

\newtheorem{obs}[thm]{Observation}

\numberwithin{equation}{section}

\def \bn {\hfill \\ \smallskip\noindent}

\def\proof{\bn {\bf Proof.} }

\def\note#1
{\marginpar

{\tiny $\leftarrow$
\par
\hfuzz=20pt \hbadness=9000 \hyphenpenalty=-100 \exhyphenpenalty=-100
\pretolerance=-1 \tolerance=9999 \doublehyphendemerits=-100000
\finalhyphendemerits=-100000 \baselineskip=6pt
#1}\hfuzz=1pt}

\newcommand{\R}{{\mathbb R}}
\newcommand{\T}{{\mathbb T}}

\newcommand{\Z}{{\mathbb Z}}

\def\B0{{\bold{0}}}

\newcommand{\hl}{\hat{\lambda}}

\catcode`\@=12

\def\Empty{}
\newcommand\oplabel[1]{
  \def\OpArg{#1} \ifx \OpArg\Empty {} \else
  	\label{#1}
  \fi}

\newcommand{\comm}[1]{}
\newcommand{\comment}[1]{}

\begin{document}

\title[]{Absence of point spectrum for the self-dual extended Harper's model}

\author{Rui Han}

\address{Department of Mathematics,
    University of California, Irvine CA, 92717}
\email{rhan2@uci.edu}

\begin{abstract}
We give a simple proof of absence of point spectrum for the self-dual extended Harper's model.
We get a sharp result which improves that of \cite{AJM} in the isotropic self-dual regime.
\end{abstract}

\maketitle

\section{Introduction}
We study the extended Harper's model on $l^2(\Z)$:
\begin{align}
(H_{{\lambda}, \alpha, \theta}u)_n=c_{\lambda}(\theta+n\alpha)u_{n+1}+\tilde{c}_{\lambda}(\theta+(n-1)\alpha)u_{n-1}+v(\theta+n\alpha)u_n,
\end{align}
where $c_{\lambda}(\theta)=\lambda_1e^{-2\pi i (\theta+\frac{\alpha}{2})}+\lambda_2+\lambda_3 e^{2\pi i (\theta+\frac{\alpha}{2})}$ and $v(\theta)=2\cos{2\pi \theta}$. $\tilde{c}_{\lambda}(\theta)=\overline{c_{\lambda}(\theta)}$ for $\theta\in \T$ and its analytic extension when $\theta\notin \T$.
We refer to ${\lambda}=(\lambda_1,\lambda_2, \lambda_3)$ as coupling constants, $\theta\in \T=[0,1]$ as the phase and $\alpha$ as the frequency.

In \cite{JM12} the authors partitioned the parameter space into the following three regions.
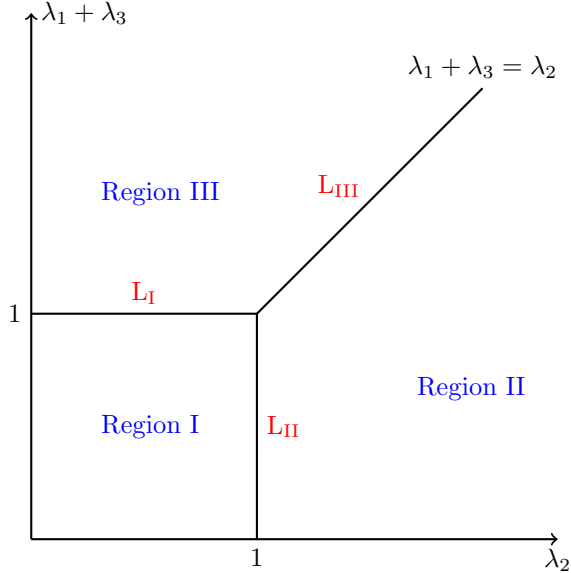
\begin{center}
\begin{tikzpicture}[thick, scale=1]
    \draw[->] (-10,-1) -- (-3,-1) node[below] {$\lambda_2$};
    \draw[->] (-10,-1) -- (-10,6) node[right] {$\lambda_1+\lambda_3$};
    \draw [ ] plot [smooth] coordinates { (- 7, 2) (-4, 5) };
    \draw [ ] plot [smooth] coordinates { (- 7, 2) (-7,-1) };
    \draw [ ] plot [smooth] coordinates { (-10, 2) (-7, 2) };

    \draw(-  4,   5) node [above] {$\lambda_1+\lambda_3=\lambda_2$};
    \draw(-  7,  -1) node [below] {$1$};
    \draw(- 10,   2) node [left]  {$1$};
    \draw(-9.2, 0.5) node [color=blue][right] {Region I};
    \draw(-  5,   1) node [color=blue][right] {Region II};
    \draw(-9.2, 3.6) node [color=blue][right] {Region III};

    \draw(-  7, 0.5) node [color=red][right] {$\mathrm{L}_{\mathrm{II}}$};
    \draw(-8.5,   2) node [color=red][above] {$\mathrm{L}_{\mathrm{I}}$};
    \draw(-5.5, 3.7) node [color=red][left] {$\mathrm{L}_{\mathrm{III}}$};

\end{tikzpicture}
\end{center}
\begin{description}
\item[Region I] $0 \leq \lambda_{1}+\lambda_{3} \leq 1, 0 < \lambda_{2} \leq 1$,
\item[Region II] $0 \leq \lambda_{1}+\lambda_{3} \leq \lambda_{2}, 1 \leq \lambda_{2} $,
\item[Region III] $\max\{1,\lambda_{2}\} \leq \lambda_{1}+\lambda_{3}$, $\lambda_2>0$.
\end{description}
According to the action of the {\it duality transformation} $\sigma: \lambda=(\lambda_1, \lambda_2, \lambda_3)\rightarrow \hat{\lambda}=(\frac{\lambda_3}{\lambda_2}, \frac{1}{\lambda_2}, \frac{\lambda_1}{\lambda_2})$, we have the following observation \cite{JM12}:
\begin{obs}
$\sigma$ is a bijective map on $0\leq \lambda_1+\lambda_3,\ 0<\lambda_2$.
\begin{itemize}
\item[(i)] $\sigma(\mathrm{I}^\circ) = \mathrm{II}^\circ$, $\sigma(\mathrm{III}^\circ) = \sigma(\mathrm{III}^\circ)$
\item[(ii)] Letting $\mathrm{L}_\mathrm{I}:=\{\lambda_1 + \lambda_3=1, 0 < \lambda_2 \leq 1\}$, $\mathrm{L}_{\mathrm{II}}:=\{0 \leq \lambda_1 + \lambda_3 \leq 1, \lambda_2 = 1\}$, and $\mathrm{L}_{\mathrm{III}}:=\{1 \leq \lambda_1 + \lambda_3 = \lambda_2\}$, $\sigma(\mathrm{L}_\mathrm{I}) = \mathrm{L}_{\mathrm{III}}$ and $\sigma(\mathrm{L}_{\mathrm{II}})= \mathrm{L}_{\mathrm{II}}$.
\end{itemize}
\end{obs}
As $\sigma$ bijectively maps $\mathrm{III} \cup \mathrm{L}_{\mathrm{II}}$ onto itself, the literature refers to $\mathrm{III} \cup \mathrm{L}_{\mathrm{II}}$ as the {\it self-dual regime}. We further divide $\mathrm{III}$ into $\mathrm{III}_{\lambda_1=\lambda_3}$ ({\it isotropic self-dual regime}) and $\mathrm{III}_{\lambda_1\neq \lambda_3}$ ({\it anisotropic self-dual regime}).

A complete understanding of the spectral properties of the extended Harper's model for a.e. $\theta$ has been established:
\begin{thm}\label{AJM1}\cite{AJM}The following Lebesgue decomposition of the spectrum of $H_{\lambda,\alpha,\theta}$ holds for a.e.$\theta$.
\begin{itemize}
\item For all Diophantine $\alpha$, for Region $\mathrm{I}$, $H_{\lambda, \alpha, \theta}$ has pure point spectrum.
\item For all irrational $\alpha$, for Regions $\mathrm{II}$, $\mathrm{III}_{\lambda_1\neq \lambda_3}$, $H_{\lambda,\alpha,\theta}$ has purely absolutely continuous spectrum.
\item For all irrational $\alpha$, for Region $\mathrm{III}_{\lambda_1=\lambda_3}$, $H_{\lambda,\alpha,\theta}$ has purely singular continuous spectrum.
\end{itemize}
\end{thm}
As pointed out in \cite{AJM}, the main missing link between \cite{JM12, JM13} and Theorem \ref{AJM1} is the following theorem, excluding eigenvalues in the self-dual regime. We say $\theta$ is $\alpha$-rational if $2\theta\in \Z\alpha+\Z$, otherwise we say $\theta$ is $\alpha$-irrational.
\begin{thm}\cite{AJM}\label{AJM2}
For all irrational $\alpha$,
\begin{itemize}
\item for $\lambda\in \mathrm{III}_{\lambda_1\neq \lambda_3}\cup \mathrm{L}_{\mathrm{II}}$, $H_{\lambda,\alpha,\theta}$ has empty point spectrum for all $\alpha$-irrational $\theta$.
\item for $\lambda\in \mathrm{III}_{\lambda_1=\lambda_3}$, $H_{\lambda,\alpha,\theta}$ has empty point spectrum for a.e. $\theta$.
\end{itemize}
\end{thm}
In \cite{AJM} the authors had to exclude more phases than $\alpha$-rational $\theta$ in the isotropic self-dual regime. 

In this paper we give a simple proof of the following theorem.
\begin{thm}\label{main}
For all irrational $\alpha$, for $\lambda\in \mathrm{III}$, $H_{\lambda,\alpha,\theta}$ has empty point spectrum for all $\alpha$-irrational $\theta$.
\end{thm}
\begin{rem}
Our result for the isotropic self-dual regime $\mathrm{III}_{\lambda_1=\lambda_3}$ is sharp. Indeed, according to Proposition 5.1 in \cite{AJM}, for $\alpha$-rational $\theta$, $H_{\lambda, \alpha, \theta}$ has point spectrum.
\end{rem}

We organize this paper in the following way: in Section 2 we include some preliminaries, in Section 3 we present two lemmas that will be used in Section 5, then we deal with $\mathrm{III}_{\lambda_1=\lambda_3}$ and $\mathrm{III}_{\lambda_1\neq \lambda_3}\cap \{\lambda_1+\lambda_3=1\}$ in Section 4 and $\mathrm{III}_{\lambda_1\neq \lambda_3}\cap \{\lambda_1+\lambda_3>1\}$ in Section 5.

\section{Preliminaries}
\subsection{Rational approximation}
Let $\{\frac{p_m}{q_m}\}$ be the continued fraction approximants of $\alpha$, then
\begin{align}\label{qnqn+1}
\frac{1}{2q_{m+1}}\leq \|q_m\alpha\|_{\T}\leq \frac{1}{q_{m+1}}.
\end{align}
The exponent $\beta(\alpha)$ is defined as follows
\begin{align}\label{defbeta}
\beta(\alpha)=\limsup_{m\rightarrow\infty}\frac{\ln{q_{m+1}}}{q_m}.
\end{align}
It describes how well is $\alpha$ approximated by rationals. 

\subsection{Self-dual extended Harper's model}
Let $|c|_{\lambda}(\theta)=\sqrt{c_{\lambda}(\theta)\tilde{c}_{\lambda}(\theta)}$ be the analytic function that coincides with $|c_{\lambda}(\theta)|$ when $\theta\in \T$. 

The presence of singularities of $c_{\lambda}(\theta)$ is explicit:
\begin{obs}(e.g. \cite{AJM})\label{sing}
The function $c_{\lambda}(\theta)$ has at most two zeros. Necessary conditions for real roots are $\lambda\in \mathrm{III}_{\lambda_1=\lambda_3}$ or $\lambda\in \mathrm{III}_{\lambda_1\neq \lambda_3}\cap \{\lambda_1+\lambda_3=\lambda_2\}$. Moreover,
\begin{itemize}
\item for $\lambda\in \mathrm{III}_{\lambda_1=\lambda_3}$, $c_{\lambda}(\theta)$ has real roots determined by
\begin{align}
2\lambda_3\cos{2\pi(\theta+\frac{\alpha}{2})}=-\lambda_2, 
\end{align}
and giving rise to a double root at $\theta=\frac{1}{2}-\frac{\alpha}{2}$ if $\lambda\in \mathrm{III}_{\lambda_1=\lambda_3}\cap \{\lambda_1+\lambda_3=\lambda_2\}$.
\item for $\lambda\in \mathrm{III}_{\lambda_1\neq \lambda_3}\cap \{\lambda_1+\lambda_3=\lambda_2\}$, $c_{\lambda}(\theta)$ has only one simple real root at $\theta=\frac{1}{2}-\frac{\alpha}{2}$.
\end{itemize}
\end{obs}
\begin{rem}\label{obsrem}
By the definition of the duality transformation $\sigma$, Observation \ref{sing} implies that $c_{\hl}(\theta)$ has singular point if and only if $\lambda\in \mathrm{III}_{\lambda_1=\lambda_3}$ or $\lambda\in \mathrm{III}_{\lambda_1\neq\lambda_3}\cap \{\lambda_1+\lambda_3=1\}$.
\end{rem}
It will be clear in Section 4 that presence of singularities of $c_{\hl}(\theta)$ indeed simplifies the proof of empty point spectrum of $H_{\lambda,\alpha,\theta}$.

\section{Lemmas}
\begin{lemma}\label{analyticexp}
For $\lambda\in \mathrm{III}_{\lambda_1\neq \lambda_3}\cap \{\lambda_1+\lambda_3>1\}$, when $\lambda_3>\lambda_1$, we have
\begin{align*}
\frac{c_{\hl}(\theta)}{|c|_{\hl}(\theta)}=e^{-2\pi i (\theta+\frac{\alpha}{2})+if(\theta)}\ \ \mathrm{and}\ \  \frac{\tilde{c}_{\hl}(\theta)}{|c|_{\hl}(\theta)}=e^{2\pi i (\theta+\frac{\alpha}{2})-if(\theta)},
\end{align*} 
for a real analytic function $f(\theta)$ on $\T$ with $\int_{\T} f(\theta)\mathrm{d}\theta =0 $.
\end{lemma}   
\begin{proof}
By the definition of $c_{\hl}(\theta)$ we have
\begin{align}
c_{\hl}(\theta)
=&\frac{\lambda_3}{\lambda_2} e^{-2\pi i (\theta+\frac{\alpha}{2})}+\frac{1}{\lambda_2}+\frac{\lambda_1}{\lambda_2} e^{2\pi i (\theta+\frac{\alpha}{2})}\\
=&\frac{\lambda_1}{\lambda_2} e^{-2\pi i (\theta+\frac{\alpha}{2})}(e^{2\pi i (\theta+\frac{\alpha}{2})}-y_+)(e^{2\pi i (\theta+\frac{\alpha}{2})}-y_-),
\end{align}
where $y_{\pm}=\frac{-1 \pm \sqrt{1-4\lambda_1\lambda_3}}{2\lambda_1}$. Note that
\begin{align}\label{ypmdef}
y_+=\overline{y_-}\ \mathrm{with}\ |y_{+}|&=|y_{-}|=\sqrt{\frac{\lambda_3}{\lambda_1}}>1,\ \ \mathrm{when}\ 1\leq 2\sqrt{\lambda_1\lambda_3},\\
y_+, y_-\in \R\ \mathrm{with}\ |y_+|&>|y_-|=\frac{2\lambda_3}{\lambda_1+\sqrt{1-4\lambda_1\lambda_3}}>1,\ \ \mathrm{when}\ \lambda_1+\lambda_3>1>2\sqrt{\lambda_1\lambda_3}.
\end{align}
Note that 
\begin{align}\label{1}
\frac{c_{\hl}(\theta)}{|c|_{\hl}(\theta)}=\sqrt{\frac{c_{\hl}(\theta)}{\tilde{c}_{\hl}(\theta)}}=e^{-2\pi i (\theta+\frac{\alpha}{2})}\sqrt{\frac{(e^{2\pi i(\theta+\frac{\alpha}{2})}-y_{+})(e^{2\pi i (\theta+\frac{\alpha}{2})}-y_{-})}{(e^{-2\pi i (\theta+\frac{\alpha}{2})}-y_{+})(e^{-2\pi i (\theta+\frac{\alpha}{2})}-y_{-})}}.
\end{align}
By (\ref{ypmdef}), we have
\begin{align}\label{arg}
\int_{\T} \arg \frac{(e^{2\pi i(\theta+\frac{\alpha}{2})}-y_{+})(e^{2\pi i (\theta+\frac{\alpha}{2})}-y_{-})}{(e^{-2\pi i (\theta+\frac{\alpha}{2})}-y_{+})(e^{-2\pi i (\theta+\frac{\alpha}{2})}-y_{-})} \mathrm{d}\theta=0,
\end{align}
and
\begin{align}\label{norm}
|\frac{(e^{2\pi i(\theta+\frac{\alpha}{2})}-y_{+})(e^{2\pi i (\theta+\frac{\alpha}{2})}-y_{-})}{(e^{-2\pi i (\theta+\frac{\alpha}{2})}-y_{+})(e^{-2\pi i (\theta+\frac{\alpha}{2})}-y_{-})}|\equiv 1.
\end{align}
Thus there exists a real analytic function $g(\theta)$ on $\T$ such that
\begin{align}
\frac{(e^{2\pi i(\theta+\frac{\alpha}{2})}-y_{+})(e^{2\pi i (\theta+\frac{\alpha}{2})}-y_{-})}{(e^{-2\pi i (\theta+\frac{\alpha}{2})}-y_{+})(e^{-2\pi i (\theta+\frac{\alpha}{2})}-y_{-})}=e^{ig(\theta)},
\end{align}
with $\int_{\T}g(\theta)\mathrm{d}\theta=0$. Taking $f(\theta)=g(\theta)/2$ yields the desired the result.
$\hfill{} \Box$
\end{proof}

\begin{lemma}\label{uniform0}
There is a subsequence $\{\frac{p_{m_l}}{q_{m_l}}\}$ of the continued fraction approximants of $\alpha$ so that 
for any analytic function $f$ on $\T$ with $\int_{\T} f(\theta)\mathrm{d}\theta=0$, we have 
\begin{align*}
\lim_{l\rightarrow \infty} f(x)+f(x+\alpha)+ \cdots +f(x+q_{m_l}\alpha-\alpha)=0
\end{align*}
uniformly in $x\in \T$.
\end{lemma}
\begin{proof}
Suppose $f$ is analytic on $|\mathrm{Im}\theta| \leq \delta_0$, then $|\hat{f}(n)|\leq ce^{-2\pi \delta_0 |n|}$ for some constant $c>0$.
\subparagraph{Case 1}
If $\beta(\alpha)=0$, then by solving the coholomogical equation we get $f(x)=h(x+\alpha)-h(x)$ for some analytic $h(x)$. Then
\begin{align*}
  &\lim_{m\rightarrow \infty}(f(x)+f(x+\alpha)+\cdots+f(x+q_m\alpha-\alpha)) \\
=&\lim_{m\rightarrow \infty}(h(x+q_m\alpha)-h(x))=0
\end{align*}
uniformly in $x$.

\subparagraph{Case 2}
If $\beta(\alpha) > 0$,
choose a sequence $m_l$ such that $q_{m_l+1}\geq e^{\frac{\beta}{2}q_{m_l}}$.
Then
\begin{align*}
        & | f(x)+f(x+\alpha)+\cdots+f(x+q_{m_l}\alpha-\alpha) | \\
      =& |\sum_{|n|\geq 1} \hat{f}(n)(1+e^{2\pi in \alpha}+\cdots+e^{2\pi i n(q_{m_l}-1)\alpha})e^{2\pi i nx}| \\
      =& |\sum_{|n|\geq 1}\hat{f}(n)\frac{1-e^{2\pi i n q_{m_l} \alpha}}{1-e^{2\pi i n \alpha}}e^{2\pi i nx}| \\
 \leq &\sum_{1\leq |n|\leq q_{m_l}-1} c \left|\frac{1-e^{2\pi i n q_{m_l} \alpha}}{1-e^{2\pi i n \alpha}}\right|
         +\sum_{|n|\geq q_{m_l}}ce^{-2\pi \delta_0 |n|}q_{m_l} \\
 \leq & c\frac{q_{m_l}^3}{q_{m_l+1}}+cq_{m_l} e^{-2\pi \delta_0 q_{m_l}} \rightarrow 0\ \ \mathrm{as}\ l\rightarrow\infty
\end{align*} 
uniformly in $x$.
$\hfill{} \Box$
\end{proof}

\section{Consequence of point spectrum}
This part follows from \cite{AJM}. We present the material here for completeness and readers' convenience.

Suppose $\{u_n\}$ is an $l^2(\Z)$ solution to $H_{{\lambda},\alpha,\theta}u=Eu$, where ${\lambda}=(\lambda_1, \lambda_2, \lambda_3)$. 
This means
\begin{equation}\label{evequation}
c_{\lambda}(\theta+n\alpha) u_{n+1}+\tilde{c}_{\lambda}{(\theta+(n-1)\alpha)}u_{n-1}+2\cos(2\pi(\theta+n\alpha))u_n=Eu_n.
\end{equation}
Let $u(x)=\sum_{n\in \Z}u_ne^{2\pi i nx}\in L^2(\T)$. 
Multiplying (\ref{evequation}) by $e^{2\pi i nx}$ and then summing over $n$, we get
\begin{equation}\label{E1}
e^{2\pi i \theta}c_{\hat{\lambda}}(x)u(x+\alpha)+e^{-2\pi i \theta}{\tilde{c}_{\hat{\lambda}}(x-\alpha)}u(x-\alpha)+2\cos2\pi x\ u(x)=\frac{E}{\lambda_2} u(x),
\end{equation}
where $\hat{\lambda}=(\frac{\lambda_3}{\lambda_2}, 1, \frac{\lambda_1}{\lambda_2})$.
If we multiply (\ref{evequation}) by $e^{-2\pi i nx}$ and sum over $n$, we get
\begin{equation}\label{E2}
e^{-2\pi i \theta}c_{\hat{\lambda}}(x)u(-x-\alpha)+e^{2\pi i \theta}\tilde{c}_{\hat{\lambda}}{(x-\alpha)}u(-x+\alpha)+2\cos2\pi x\ u(-x)=\frac{E}{\lambda_2} u(-x).
\end{equation}
Thus writing (\ref{E1}), (\ref{E2}) in terms of matrices, we get
\begin{align}\label{conju1}
&\frac{1}{c_{\hat{\lambda}}(x)}
\left(\begin{array}{cc}
\frac{E}{\lambda_2}-2\cos 2\pi x & -\tilde{c}_{\hat{\lambda}}(x-\alpha) \\
c_{\hat{\lambda}}(x)    &  0
\end{array}\right)
\left(\begin{array}{cc}
u(x) & u(-x) \notag\\
e^{-2\pi i \theta}u(x-\alpha)   &  e^{2\pi i \theta}u(-(x-\alpha)) 
\end{array}\right) \\
=& \left(\begin{array}{cc}
u(x+\alpha) & u(-(x+\alpha)) \\
e^{-2\pi i \theta}u(x)   &  e^{2\pi i \theta}u(-x)
\end{array} \right)
\left(\begin{array}{cc}
e^{2\pi i \theta} & 0 \\
0 &  e^{-2\pi i \theta}
\end{array}\right)
\end{align}
Let $M_{\theta}(x)\in L^2(\T)$ be defined by
\begin{equation*}
M_{\theta}(x)=\left( \begin{array}{cc}
u(x) & u(-x)\\
e^{-2\pi i \theta}u(x-\alpha)   &  e^{2\pi i \theta}u(-(x-\alpha)) 
\end{array}\right).
\end{equation*}
Let
\begin{equation*}
A_{\hat{\lambda},E/{\lambda_2}}(x)=\frac{1}{c_{\hat{\lambda}}(x)}
\left(\begin{array}{cc}
\frac{E}{\lambda_2}-2\cos{2\pi x} & -\tilde{c}_{\hat{\lambda}}(x-\alpha) \\
c_{\hat{\lambda}}(x)    &  0
\end{array}\right)
\end{equation*}
be the transfer matrix associated to $H_{\hat{\lambda},\alpha, \theta}$
and 
\begin{equation*}
R_{\theta}=\left(\begin{array}{cc}
e^{2\pi i \theta} & 0 \\
0 &  e^{-2\pi i \theta}
\end{array}\right)
\end{equation*}
be the constant rotation matrix.
Then (\ref{conju1}) becomes
\begin{align}\label{conju2}
A_{\hat{\lambda},E}(x) M_{\theta}(x)=M_{\theta}(x+\alpha)R_{\theta}.
\end{align}
Taking determinant, we have the following proposition.
\begin{prop}\cite{AJM}
If $\theta$ is $\alpha$-irrational, then
\begin{align}\label{det=}
|\det M_{\theta}(x)|=\frac{b}{|c|_{\hat{\lambda}}(x-\alpha)}
\end{align}
for some constant $b > 0$ and a.e.$x \in \T$.
\end{prop}

\section{Regions $\mathrm{III}_{\lambda_1=\lambda_3}$ and $\mathrm{III}_{\lambda_1\neq \lambda_3}\cap \{\lambda_1+\lambda_3=1\}$}
We will show the following lemma.
\begin{lemma}
If $\theta$ is $\alpha$-irrational, then for $\lambda\in\mathrm{III}_{\lambda_1=\lambda_3}$ or $\lambda\in\mathrm{III}_{\lambda_1\neq \lambda_3}\cap \{\lambda_1+\lambda_3=1\}$, $H_{\lambda, \alpha, \theta}$ has no point spectrum.
\end{lemma}
\begin{proof}
According to Remark \ref{obsrem}, we have $c_{\hat{\lambda}}(x_0)=0$ for some $x_0\in \T$.
Note that presence of singularity implies $\frac{1}{c_{\hat{\lambda}}(x)}\notin L^1(\T)$.
Thus by (\ref{det=}), $\det M_{\theta}(x)\notin L^1(\T)$. This contradicts with $M_{\theta}(x)\in L^2(\T)$.
$\hfill{} \Box$
\end{proof}

\section{Regions $\mathrm{III}_{\lambda_1\neq \lambda_3}\cap \{\lambda_1+\lambda_3>1\}$}
Without loss of generality, we assume $\lambda_3>\lambda_1$. Fix $\theta$. Denote $\det M_{\theta}(x)=g(x)$ for simplicity.
\begin{lemma}
If $\theta$ is $\alpha$-irrational, then $H_{\lambda,\alpha,\theta}$ has no point spectrum in the anisotropic self-dual region.
\end{lemma}
\begin{proof}
Taking determinant in (\ref{conju2}), we get:
\begin{align*}
\frac{\tilde{c}_{\hat{\lambda}}(x-\alpha)}{c_{\hat{\lambda}}(x)}g(x)=g(x+\alpha). 
\end{align*}
This implies 
\begin{align}\label{gkexp}
g(x+k\alpha)=\frac{\tilde{c}_{\hl}{(x+k\alpha-2\alpha)} \cdots \tilde{c}_{\hl}{(x)}\  \tilde{c}_{\hl}{(x-\alpha)}}{c_{\hl}(x+k\alpha-\alpha) \cdots c_{\hl}(x+\alpha) c_{\hl}(x) } g(x).
\end{align}
Taking $k=q_{m_l}$, as in Lemma \ref{uniform0}, on one hand, since $g(x)$ is an $L^1$ function, as the determinant of an $L^2$ matrix, and $\lim_{l\rightarrow\infty}\|q_{m_l}\alpha\|_{\T}= 0$, we have
\begin{align*}
\lim_{l\rightarrow \infty} \|g(x+q_{m_l}\alpha)-g(x)\|_{L^1}=0.
\end{align*}
By (\ref{gkexp}), this implies
\begin{align}\label{AI1}
  0=\lim_{l\rightarrow \infty} \|g(x+q_{m_l}\alpha)-g(x)\|_{L^1}
=\lim_{l\rightarrow \infty}  \int \left|1-\frac{\prod_{j=-1}^{q_{m_l}-2}\tilde{c}_{\hl}(x+j\alpha)}{\prod_{j=0}^{q_{m_l}-1}c_{\hl}(x+j\alpha)}\right| \cdot |g(x)| \mathrm{d}x.
\end{align}
On the other hand, by Lemma \ref{analyticexp}
\begin{align}\label{Al2}
        & \lim_{l\rightarrow \infty} \int \left|1-\frac{\prod_{j=-1}^{q_{m_l}-2}\tilde{c}_{\hl}(x+j\alpha)}{\prod_{j=-1}^{q_{m_l}-1}c_{\hl}(x+j\alpha)}\right|\cdot  |g(x)|\mathrm{d}x \notag\\
     = & \lim_{l\rightarrow \infty} \int \left| 1-\frac{|c|_{\hl}(x-\alpha)}{|c|_{\hl}(x+q_{m_l}\alpha-\alpha)} 
      e^{-i\left(\sum_{j=-1}^{q_{m_l}-2}f(x+j\alpha)+\sum_{j=0}^{q_{m_l}-1}f(x+j\alpha)\right)} e^{4\pi i q_{m_l} x} e^{2\pi i q_{m_l} (q_{m_l}-1)\alpha}\right|\cdot        |g(x)| \mathrm{d}x \notag\\
\geq & \liminf_{l\rightarrow \infty} \left( \int |1-e^{4\pi i q_{m_l} x+2\pi i q_{m_l}^2\alpha}| |g(x)| \mathrm{d}x \right.\notag\\
        &\ \ \ \ \ \ \ \     \left. -\int \left| 1-\frac{|c|_{\hl}(x-\alpha)}{|c|_{\hl}(x+q_{m_l}\alpha-\alpha)} 
                               e^{-i\left(\sum_{j=-1}^{q_{m_l}-2}f(x+j\alpha)+\sum_{j=0}^{q_{m_l}-1}f(x+j\alpha)\right)} e^{-2\pi i   
                               q_{m_l}\alpha}\right|\cdot |g(x)| \mathrm{d}x  \right)  \notag\\
     := &\liminf_{l\rightarrow \infty} (I_1-I_2).
\end{align}
Combining the fact $\|q_{m_l}\alpha\|_{\T}\rightarrow 0$ with Lemma \ref{uniform0}, we get pointwise convergence,
\begin{align*}
\frac{|c|_{\hl}(x-\alpha)}{|c|_{\hl}(x+q_{m_l}\alpha-\alpha)} 
                               e^{-i\left(\sum_{j=-1}^{q_{m_l}-2}f(x+j\alpha)+\sum_{j=0}^{q_{m_l}-1}f(x+j\alpha)\right)} e^{-2\pi i   
                               q_{m_l}\alpha}\rightarrow 1\ \ \mathrm{as}\ l\rightarrow\infty.
\end{align*}
Then by dominated convergence theorem, we get $\lim_{l\rightarrow \infty}I_2=0$.
Then (\ref{Al2}) implies that for any small constant $\delta>0$,
\begin{align*}
        &\lim_{l\rightarrow \infty} \|g(x+q_{m_l}\alpha)-g(x)\|_{L^1}\\
\geq &\liminf_{l\rightarrow \infty}I_1\\
\geq &\liminf_{l\rightarrow \infty}  \int_{\|2q_{m_l}x+q_{m_l}^2\alpha\|_{\T}\geq\delta} 4\delta |g(x)| \mathrm{d}x,
\end{align*}
where $|\{x:\|2q_{m_l}x+q_{m_l}^2 \alpha\|\geq\delta\}| \triangleq |F_{m_l,\delta}| = 1-2\delta$.
Thus

\begin{align*}
        &\lim_{l\rightarrow \infty} \|g(x+q_{m_l}\alpha)-g(x)\|_{L^1}\\
     \geq &\liminf_{l\rightarrow \infty} ( 4\delta\|g\|_{L^1}-4\delta\int_{F^c_{{m_l},\delta}} |g(x)| \mathrm{d}x ) \\
\geq &\liminf_{l\rightarrow \infty} ( 4\delta\|g\|_{L^1}-8\delta^2 \|g\|_{L^{\infty}} ).    
\end{align*}
By (\ref{det=}) $|g(x)|=\frac{b}{|c|_{\hl}(x-\alpha)}$ for some constant $b>0$, thus $ \|g\|_{L^1}$, $ \|g\|_{L^{\infty}}$ are positive finite numbers, so one can choose $\delta \sim 0$ such that 
$4\delta \|g\|_{L^1}-8\delta^2 \|g\|_{L^{\infty}}$ is strictly positive. This contradicts with (\ref{AI1}). $\hfill{} \Box$
\end{proof}

\section*{Acknowledgement}
This research was partially supported by the NSF DMS-1401204. I would like to thank Svetlana Jitomirskaya for useful discussions.

\bibliographystyle{amsplain}

\end{document}